\documentclass[a4paper]{amsart}

\usepackage{float}
\usepackage{amssymb, amsfonts, amsmath, amscd, mathrsfs, mathtools, bbm,subfigure,color,comment,xcolor}

\newtheorem{proposition}{Proposition}

\newtheorem{corollary}{Corollary}
\newtheorem{lemma}{Lemma}

\newcommand{\N}{\ensuremath{\mathbb{N}}}
\newcommand{\Z}{\ensuremath{\mathbb{Z}}}
\newcommand{\I}{\ensuremath{\mathbb{I}}}

\newcommand{\R}{\ensuremath{\mathbb{R}}}
\newcommand{\C}{\ensuremath{\mathbb{C}}}
\newcommand{\cc}{\mathfrak{c}}
\newcommand{\E}{\ensuremath{\mathbb{E}}}
\renewcommand{\P}{\ensuremath{\mathbb{P}}}

\newcommand{\ind}[1]	{\ensuremath{\mathbbm{1}_{\left\{ {#1} \right\}}}}

\date{}

\begin{document}

\colorlet{blue}{black}

\title[Departure from the  $M/M/\infty$ system]{On the number of departures from the $M/M/\infty$ queue in a finite time interval}

\author{Fabrice Guillemin}
\address{Orange Labs, CNC/NARA, 2 Avenue Pierre Marzin, 22300 Lannion, France}


\begin{abstract}
In this paper, we analyze the number of departures from an initially empty  $M/M/\infty$ system in a finite time interval. We observe the system during an exponentially distributed period of time starting from the time origin. We then consider the absorbed Markov chain describing the number of arrivals and departures in the system until the observer leaves the system, triggering the absorption  of the Markov chain. The generator of the absorbed Markov chain induces a selfadjoint operator in some Hilbert space. The use of spectral theory then allows us to compute the Laplace transform of several transient characteristics of the $M/M/\infty$ system (namely, the number of transitions of the Markov chain until absorption, the number of departures from the system, etc.). \textcolor{blue}{The analysis is extended to the finite capacity $MM/\cc/\cc$ system for some finite integer $\cc$}.
\end{abstract}

\maketitle

\section{Introduction}

The $M/M/\infty$ model is a fundamental queuing system, which has  applications in many different domains such as computer science (see for instance \cite{MorrisonMMinf} for the performance of hashing with lazy deletion), telecommunications networks (notably for modeling open loop statistical multiplexing of bulk data transfers \cite{FGAS}), etc. A key characteristic of the $M/M/\infty$ model is that the time evolution of the number of customers in the system can be described by a birth and death process, which can be analyzed by means of spectral theory \cite{reed}. Specifically, the generator of this Markov process is a tridiagonal matrix inducing a selfadjoint operator in an ad-hoc Hilbert space. 

The analysis of the $M/M/\infty$ via spectral theory  dates back to the 1950s in the seminal papers by Karlin and McGregor, see \cite{KMcG} for the $M/M/\infty$ system and some related models as well as \cite{KMcGBDP} for the analysis of birth and death processes, notably giving an expression for transition probabilities by means of the associated orthogonal polynomial system and the spectral measure. The connection between birth and death processes and continued fractions has been investigated in many papers. Let us just mention that in \cite{flajoletbdp}, special attention is paid to the connection between birth and death processes and lattice path combinatorics; this connection in turn yields  results on transient characteristics of birth and death processes (and the $M/M/\infty$ system in particular) by exploiting and generalizing results obtained by Flajolet in \cite{flajoletlp}.

The $M/M/\infty$ system is described in standard textbooks on queuing theory  \cite{grossharris,Klein0,Ross97}. \textcolor{blue}{Considering an $M/M/\infty$ system with arrival rate $\rho$ and mean service time equal to unity,}  it  is   known that the probability mass function of the number of customers in the system in the stationary regime is Poisson with mean $\rho$. Specifically, if $N$ denotes the number of customers in the system in the stationary regime, then
$$
\P(N=n) = \frac{\rho^n}{n!} e^{-\rho}.
$$
In addition, if the system is initially empty, the number $N(t)$ of customers in the system at time $t$ is \cite{grossharris}
\begin{equation}
    \label{distibN}
    \P(N(t)=m~|~ N(0)=0) = e^{-\rho(1- e^{-t})} \frac{(\rho (1-e^{-t}))^m}{m!},
\end{equation}
which is a Poisson law with mean $\rho(1- e^{-t})$. This result was established in \cite{KMcG} by solving the Markov forward equations by means of spectral theory. The spectral measure is Poisson with mean $\rho $ and the generator of the Markov process describing the number of customers in the system over time has eigenvectors, which can be expressed by means of Charlier polynomials \cite{erdelyi} forming the polynomial orthogonal system associated with the $M/M/\infty$ model.

The duration of an excursion as well as the area of an excursion above a given threshold of the occupation process of an $M/M\infty$ system have been studied in \cite{guillemin_pinchon_1998,guillemin_pinchon_1999}. The analysis reveals the key role played by associated Charlier polynomials, which  satisfy the same recurrence relation as Charlier polynomials but starting from a certain index (namely the excursion threshold).  The relationships between orthogonal polynomial systems and their associated orthogonality measure, spectral theory, continued fractions, etc. are clearly explained in \cite{ismail,chihara}. In this paper, we shall use these reference books for studying the number of departure from the $M/M/\infty$ system in a finite time interval.

While many transient characteristics of an $M/M/\infty$ system are perfectly known in the queuing literature, the number of departures from the system in a given time interval \textcolor{blue}{is more rarely considered in  the technical literature. In this paper, we consider this random variable for an initially empty system. This latter assumption is motivated by the fact that if there are $N_0(0)= n_0$ customers in the system at time $t=0$, then the number of departures $D_0(t)$ from the system at time $t$  among these customers is simply a Bernoulli random variable with mean $n_0(1-e^{-t})$, that is, for $0\leq k \leq n_0$
$$
\P(D_0(t) = k) = \binom{n_0}{k} (1-e^{-t})^k e^{-(n_0-k)t }.
$$
The most challenging issue is to compute the number of departures $D(t)$ when starting from an empty system as it involves the transient behavior of the system. (It is also worth noting  that the random variables $D_0(t)$ and $D(t)$ are independent.)}

\textcolor{blue}{Generally  speaking, when we consider an initially empty $M/G/\infty$ queue with distribution $G(t)$ for the service time $S$ (i.e., $\P(S\leq t) = G(t)$) and with arrival rate $\rho$, the number $\mathcal{D}(t)$ of departures from the queue at time $t$ is given by 
\begin{equation}
    \label{mathcalD}
  \P(\mathcal{D}(t)=k) = \sum_{n=k}^\infty \binom{n}{k} p(t)^k(1-p(t))^{n-k} \frac{(\rho t)^n}{n!}e^{-\rho t} = \frac{(\rho t p(t))^k}{k!} e^{-\rho t p(t)},
\end{equation}
where 
\begin{equation}
    \label{defpt}
    p(t) = \frac{1}{t}\int_0^t G(t-u)du,
\end{equation}
showing that the random variable $\mathcal{D}(t)$ is Poisson with mean $\rho t p(t)$. As a matter of fact, since the arrival process is Poisson with intensity $\rho$,  the number of arrivals $A(t)$ in the time interval $t$ is such that
$$
\P(A(t) = n) = \frac{(\rho t)^n}{n!}e^{-\rho t}.
$$
In addition, because the arrival process is Poisson, each  customer arrives at a uniformly distributed random time between 0 and $t$. Assuming that a customer arrives at time $u$, this customer leaves the system before time $t$ if its service time is less than $t-u$, with probability $G(t-u)$. Hence, the probability that an arbitrary arriving customer leaves the system before $t$ is given by $p(t)$. Since there are no interactions between customers and conditioning on the number of arriving customers $A(t)=n$, the number of customers which arrive in $(0,t)$ and leave the system before $t$ is Bernoulli with mean 
$np(t)$. By deconditioning on $A(t)$, Equation~\eqref{mathcalD} easily follows. In the specific case of the $M/M/\infty$ queue with unit mean service time, we have
$$
p(t) = \frac{1}{t}(e^{-t}-1+t)
$$
and then the probability mass function of the number $D(t)$ of departures in  $(0,t)$ is given by
\begin{equation}
    \label{distD}
  \P({D}(t)=k) 
  = \frac{\rho^k (t-1+e^{-t})^k}{k!} e^{\rho(1-t-e^{-t})};
\end{equation}
the random variable $D(t)$ is hence Poisson with mean $\rho (t-1+e^{-t})$.
}

\textcolor{blue}{In this paper, we show how this result can be recovered by using the spectral properties of the system. For this purpose, we adopt the same approach as in \cite{guillemin_quintuna20}.} We introduce an observer, which observes the system for an exponentially distributed duration of time with mean $1/\sigma$ and we study the Markov chain describing the number of customers in the $M/M/\infty$ system until the observer leaves the system.  This leads us to study a discrete-time Markov chain, which describes the number of arrivals and departures in the system and which is absorbed when the observer leaves. This allows us to derive the Laplace transforms of several transient characteristics, in particular that of $\P(D(t)=m)$ for some integer $m$.

This paper is organized as follows: In Section~\ref{model}, we describe the model and introduce the absorbed Markov chain of interest. In Section~\ref{spectral}, we study the spectral properties of the transition matrix of the absorbed Markov chain. The Laplace transforms of the probability mass functions of the transient characteristics   are derived in Section~\ref{transient}. \textcolor{blue}{We apply the same analysis framework to the finite capacity $M/M/\cc/\cc$ system in Section~\ref{mmcc}.} 
Some concluding remarks are presented in Section~\ref{conclusion}. 

\section{Model description and preliminary results}
\label{model}

\subsection{Notation}
We consider an $M/M/\infty$ queue with arrival rate $\rho$ and unit service rate; the system is empty at time $t=0$. We denote by $N(t)$ the number of customers in the queue at time $t$. We further introduce an observer, which  observes the $M/M/\infty$ queue  during an exponentially distributed  period of time with mean $1/\sigma$  for some $\sigma >0$. 

We consider the system composed of the $M/M/\infty$ queue and the observer and we introduce the discrete-time process $(\mathfrak{n}_k(\sigma))$ describing the number of customers in the $M/M/\infty$ queue (the observer is not included); $\mathfrak{n}_k(\sigma)$ is the number of customers in the queue at the $k$th event corresponding either to a customer arrival, or a service completion or the departure of the observer from the system. When the observer leaves the system, the process $(\mathfrak{n}_k(\sigma))$ is absorbed in some state, denoted by $-1$. The index $k$ is thus the number of departures  or arrivals  before the observer leaves the system or equivalently before the process $(\mathfrak{n}_k(\sigma))$ gets absorbed. Because the $M/M/\infty$ queue is supposed to be initially empty, we have $\mathfrak{n}_0(\sigma)=0$.

The state space of the  discrete-time Markov chain $(\mathfrak{n}_k)$ is $\{-1,0, 1,2, \ldots\}$  with transition matrix $\mathcal{A}(\sigma)$ given by
$$
\mathcal{A}(\sigma)= \begin{pmatrix}
1 & 0 & 0 & 0 & 0 & \ldots \\
\frac{\sigma}{\sigma+\rho} & 0 & \frac{\rho}{\sigma+\rho} & 0 & 0 & \ldots \\
\frac{\sigma}{1+\rho+\sigma} &\frac{1}{1+\rho+\sigma} & 0 &\frac{\rho}{1+\rho+\sigma}& 0 & \ldots \\
\frac{\sigma}{2+\rho+\sigma} & 0 &\frac{2}{2+\rho+\sigma}  & 0 &\frac{\rho}{2+\rho+\sigma}  & \ldots \\
\vdots & \vdots &  & &  & \ldots \\
\end{pmatrix}.
$$
The non-zero coefficients of the matrix $\mathcal{A}(\sigma)$ are given by $\mathcal{A}_{-1,-1}(\sigma)=1$ (the state $-1$ being absorbing) and for $n \geq 1$
\begin{equation*}
\mathcal{A}_{n,-1}(\sigma)= \frac{\sigma}{n+\sigma+ \rho},\,  \mathcal{A}_{n,n-1}(\sigma)= \frac{n}{n+\sigma+\rho}, \mbox{ and }\mathcal{A}_{n,n+1}(\sigma) = \frac{\rho}{n+\sigma+\rho}
\end{equation*}

The sub-matrix $A(\sigma)$ of $\mathcal{A}(\sigma)$ obtained by deleting the first row and the first column of matrix $\mathcal{A}(\sigma)$ is a tridiagonal matrix with coefficients $a_{n,m}$ for $n,m \geq 0$. (We keep the indices ranging from 0 to infinity instead from 1; this is motivated by  recurrence relations appearing in the following.). The only ones which are non-zero are given for  $n \geq 0$ by
$$
a_{n,n+1}(\sigma)= \frac{\rho}{n+\sigma+\rho} \mbox{ and } a_{n,n-1}(\sigma)= \frac{n}{n+\sigma+\rho},
$$
with the convention $a_{-1,0}(\sigma)=0$. This matrix is sub-stochastic, and gives the transition probabilities of the Markov chain $(\mathfrak{n}_k)$ before absorption.

\subsection{Preliminary results}

Let $e_n$ be the column vector with all entries equal to 0 except the $n$th one equal to 1. Then, for $m \geq 0$, we have 
$$
\P(\mathfrak{n}_k(\sigma)= m) = {}^te_0A(\sigma)^{k} e_m,
$$
where ${}^te_n$ is the row vector equal to the transpose of the column vector $e_n$. The probability that the observer  leaves the system at stage $k\geq 1$ while there are $m$ customers, is equal to
$$
\frac{\sigma}{m+\sigma+\rho} {}^te_0A(\sigma)^{k-1} e_m.
$$

Let $\nu(\sigma)$ denote the number of customers in the $M/M/\infty$ queue when the observer leaves the system and $\kappa(\sigma)$ be the time  at which the observer leaves the system. We have  for $k \geq 1$ and $m \geq 0$
\begin{equation}
    \label{kappamu}
\P(\kappa(\sigma) = k, \nu(\sigma) = m) =\frac{\sigma}{m+\sigma+\rho} {}^te_0 A(\sigma)^{k-1} e_m .
\end{equation}

The marginal distributions are given by for $m \geq 0$
\begin{multline}
\label{eqnu}
\P(\nu(\sigma)=m) =  \frac{\sigma}{m+\sigma +\rho}\sum_{k=1}^\infty  {}^te_0A(\sigma)^{k-1} e_m \\ =  \frac{\sigma}{m+\sigma +\rho} {}^te_0(\I-A(\sigma))^{-1} e_m ,
\end{multline}
where $\I$ is the identity matrix with zero coefficients except the diagonal ones equal to 1, and
 for $k \geq 1$
$$
\P(\kappa(\sigma) = k) = \sum_{m=0}^\infty  \frac{\sigma}{m+\sigma +\rho} {}^te_0A(\sigma)^{k-1} e_m .
$$

Let $\mathfrak{a}(\sigma)$ and $\mathfrak{d}(\sigma)$ respectively  denote the number of arrivals and departures in the $M/M/\infty$ queue, while  the  observer is in the system.  We have the following conservation equations:
\begin{equation}
    \label{conservation}
    \mathfrak{a}(\sigma) + \mathfrak{d}(\sigma) = \kappa(\sigma)-1 \mbox{  and  } \mathfrak{a} (\sigma)- \mathfrak{d}  (\sigma) = \nu(\sigma), 
\end{equation}
so that 
\begin{equation}
    \label{conservationbis}
   \mathfrak{a}(\sigma) = \frac{\kappa(\sigma)+\nu(\sigma)-1}{2} \mbox{  and   }\mathfrak{d} (\sigma) = \frac{\kappa(\sigma) - \nu(\sigma) -1}{2}. 
\end{equation}

The variable $\mathfrak{a}(\sigma)$ describes the number of arrivals at the $M/M/\infty$ queue during an exponential duration with mean $1/\sigma$. It is clear that
\begin{equation}
\label{dista}
\P(\mathfrak{a}(\sigma)=m) = \int_0^\infty \frac{(\rho t)^m}{m!} \sigma e^{-(\rho+ \sigma) t} dt =\frac{\sigma \rho^m}{(\rho+\sigma)^{m+1}},
\end{equation}
since the number of arrivals $A(t)$ in a time interval of length $t$ has a Poisson probability mass function with mean $\rho t$, i.e.,
$$
\P(A(t) = m ) =  \frac{(\rho t)^m}{m!} e^{-\rho t}.
$$
Note that 
\begin{equation}
    \label{meana}
    \E(\mathfrak{a}(\sigma))= \frac{\rho}{\sigma}.
\end{equation}

The random variable $\nu(\sigma)$ is the number of customers  in the queue when the observer leaves the system and we have
\begin{equation}
    \label{nuNt}
   \P(\nu(\sigma)=m) = \int_0^\infty \P(N(t) = m~|~N(0)=0)  \sigma e^{- \sigma t} dt, 
\end{equation}
where $N(t)$ is the number of customers in the queue at time $t$.

In the following, we shall give a representation of $\P(\nu(\sigma)=m)$ by means of Charlier polynomials \cite{erdelyi}. It is worth noting that since $\E(N(t)) = \rho (1-e^{-t})$ (from Equation~\eqref{distibN}), we have
\begin{equation}
    \label{meannu}
    \E(\nu(\sigma)) = \frac{\rho}{\sigma+1}.
\end{equation}

While the random variables $\nu(\sigma)$ and $\mathfrak{a}(\sigma)$ as well as $A(t)$ and $N(t)$ are known, their correlation structure (namely, their joint probability mass  functions) is less investigated in the literature. For the random variables $\mathfrak{d}(\sigma)$ and $D(t)$, we shall use the fact that 
\begin{equation}
    \label{defDt}
   \P(\mathfrak{d}(\sigma)=m)   = \int_0^\infty   \P(D(t) =m~|~N_0=0)  \sigma e^{-\sigma t} dt.
\end{equation}
It is also worth noting that by using Equations~\eqref{conservation}, \eqref{meana} and \eqref{meannu}
\begin{equation}
\label{meand}
\E(\mathfrak{d}(\sigma))= \E(\mathfrak{a}(\sigma)) - \E(\nu(\sigma)) = \frac{\rho}{\sigma(\sigma+1)}
\end{equation}
and then
\begin{equation}
    \label{meankappa}
    \E(\kappa(\sigma))=  \frac{\rho+(\sigma+1)(\sigma+\rho)}{\sigma(\sigma+1)}.
\end{equation}

To compute the probability mass  function of the random variable $\mathfrak{d}(\sigma)$, we use the orthogonality structure associated with the $M/M/\infty$ queue, already known to Karlin and McGregor \cite{KMcG}. In particular, the resolvent  $(z\I-A(\sigma))^{-1}$ of the infinite matrix $A(\sigma)$  as well as the powers of matrix $A(\sigma)$ play a central role in the computations of the probability mass functions of random variables $\kappa(\sigma)$ and $\mathfrak{d}(\sigma)$. 

\section{Spectral properties of matrix $A$}
\label{spectral}

To compute the resolvent $(z\I-A(\sigma))^{-1}$, we prove that the infinite matrix $A(\sigma)$ induces an operator, which is self-adjoint in an appropriate Hilbert space. We then determine the spectrum of this operator and use the spectral identity \cite{reed}. To determine the spectrum of the operator $A(\sigma)$, we use for $a>0$  the Charlier polynomials  $(C_n(x;a))$ satisfying the following recursion: $C_{-1}(x;a)=0$,  $C_{0}(x;a)=1$ and for $n\geq 0$
\begin{equation}
\label{recPC}
a C_{n+1}(x;a)+(x-n-a)C_n(x;a)+nC_{n-1}(x;a)=0.
\end{equation}
It is worth noting that they satisfy the following symmetry relation: for integers $n$ and $x$
\begin{equation}
    \label{symCharlier}
    C_n(x;a)=C_x(n;a).
\end{equation}
\textcolor{blue}{Note that for instance $C_1(x;a) = \frac{a-x}{a}$ and $C_2(x;a) = \frac{a^2-(2a+1)x+x^2}{a^2}$.}

{\color{blue} 
We define the  Charlier polynomials of the second kind  $C^*_n(x;a)$, $a>0,$ by the same  recursion~\eqref{recPC} but with the initial conditions: $C^*_0(x;a)=0$ and $C^*_1(x;a) = -\frac{1}{a}$. Note that $C^*_2(x;a) = \frac{x-a-1}{a^2}$. It is known in the technical literature \cite{ismail} that 
\begin{equation}
    \label{PCcf}
  \lim_{n\to \infty} \frac{C^*_n(x;a)}{C_n(x;a)}= \frac{1}{x} \Phi(1,-x+1;-a),
\end{equation}
where $\Phi(\alpha,\gamma;z)$ is the Kummer function \cite{Abramowitz} defined by
$$
\Phi(\alpha,\beta;z) = \sum_{n=0}^\infty \frac{(\alpha)_n}{(\beta)_n}\frac{z^n}{n!},
$$
where the Pochammer symbol $(z)_n= z \ldots (z+n-1)$, and with integral representation
$$
\Phi(\alpha,\beta;z) = \frac{\Gamma(\beta)}{\Gamma(\beta-\alpha)\Gamma(\alpha)}\int_0^1 t^{\alpha-1}(1-t)^{\beta-\alpha-1} e^{z t}dt
$$
for $\Re(\beta)>\Re(\alpha)>0$, $\Gamma(z)$ denoting the Eurler's Gamma function.

The polynomials $a^nC_n(-z;a)$ and $-a^nC^*_n(-z;a)$ are the successive denominators and numerators of the continued fraction \cite{ismail}
$$
\chi_0(z;a) =  \cfrac{1}{z+a -\cfrac{a}{z+a+1 -\cfrac{2 a}{z+a+2 -\cdots}}}= \frac{1}{z}\Phi(1,1+z;-a).
$$
}

The Charlier polynomials are orthogonal with respect to the Poisson measure $d\mathcal{P}_a(x)$ on $\mathbbm{\R}$ with mean $a$, that is, the discrete  measure with atoms at points $n = 0, 1, 2, \ldots$ and with mass
$$
\frac{a^n}{n!}e^{-a}
$$
at point $n$.  {\color{blue}By using  \cite[Theorem~12.11b]{Henrici} on Stieltjes fractions, we have the relation 
\begin{equation}
    \label{henrici_charlier}
 \int_{0}^\infty \frac{C_m(x;a) }{z+x}d \mathcal{P}_a (x) = C_m(-z;a) \chi_0(z;a) + C^*_m(-z;a).
 \end{equation}
}
Finally, the exponential generating function of the Charlier polynomials is given by
\begin{equation}
    \label{PCgen}
\mathcal{C}(x;a;z) \stackrel{def}{=}\sum_{n=0}^\infty C_n(x;a)\frac{z^n}{n!} = e^z\left(1-\frac{z}{a}  \right)^x.    
\end{equation}

\subsection{Self-adjointness properties}

By considering the Hilbert space
$$
H(\sigma)=\left\{f\in \R^\N : \sum_{n=0}^\infty f_n^2 \pi_n(\sigma)<\infty \right\},
$$
where 
\begin{equation}
\label{defpin}
\pi_n(\sigma) = \frac{\sigma+\rho+n}{\sigma+\rho}\frac{\rho^n}{n!},
\end{equation}
we show that the matrix $A(\sigma)$ defines a selfadjoint operator when the Hilbert space $H(\sigma)$ is equipped with the scalar product
$$
(f,g) = \sum_{n=0}^\infty f_n g_n \pi_n(\sigma)
$$
and the norm
$$
\| f\| = \sqrt{(f,f)} = \sqrt{\sum_{n=0}^\infty f_n^2 \pi_n(\sigma)}.
$$
The Hilbert space $H(\sigma)$ is introduced because the parameters $\pi_n$ satisfy the reversibility  property with the coefficients $a_{n,m}$ of the matrix $A(\sigma)$
\begin{equation}
    \label{reversecond}
    a_{n,n+1}(\sigma)\pi_n(\sigma) = a_{n+1,n}(\sigma)\pi_{n+1}(\sigma),
\end{equation}
making the  matrix $A(\sigma)$ symmetric.

The infinite matrix $A(\sigma)$ induces in $H(\sigma)$ an operator that we also denote by $A(\sigma)$. By using the same arguments as in~\cite{FGBJ}, we can easily prove the following lemma, where we use the norm of the operator $A(\sigma)$ defined by
$$
\|A(\sigma)\| = \sup_{f\in H : \|f\|<1}|(A(\sigma)f,f)|;
$$
by definition, the operator $A(\sigma)$ is bounded if $\|A(\sigma)\|<\infty$.

\begin{lemma}
The operator $A(\sigma)$ is symmetric and bounded in $H(\sigma)$, hence self-adjoint.
\end{lemma}
\begin{proof}
The symmetry of $A(\sigma)$ is straightforward by using Equation~\eqref{reversecond}.

For $f\in H(\sigma)$, we have
$$
 (A(\sigma)f,f) = {2} \sum_{n=0}^\infty a_{n,n+1}(\sigma) f_nf_{n+1}\pi_n(\sigma) = \frac{2 \rho}{\sigma+\rho} \sum_{n=0}^\infty \frac{\rho^n}{n!}  f_nf_{n+1}  
$$
By Schwarz inequality
$$
 \left| (A(\sigma)f,f) \right| \leq  \frac{2 \rho}{\sigma+\rho} \sqrt{\sum_{n=0}^\infty \frac{\rho^n}{n!} f^2_n}\sqrt{ \sum_{n=0}^\infty  \frac{\rho^n}{n!}   f^2_{n+1} }
 $$
We clearly have
$$
 \sum_{n=0}^\infty \frac{\rho^n}{n!} f^2_n \leq \|f\|^2.
$$
Moreover,
$$
\sum_{n=0}^\infty  \frac{\rho^n}{n!}   f^2_{n+1}= \frac{\sigma+\rho}{\rho} \sum_{n=0}^\infty  \frac{(n+1)}{\sigma+\rho+n+1}   f^2_{n+1}\pi_{n+1}(\sigma) \leq \frac{\sigma+\rho}{\rho} \|f\|^2.
$$
Hence,
$$
 \left| (A(\sigma)f,f) \right| \leq 2\sqrt{\frac{\rho}{\sigma+\rho}} \|f\|^2.
$$
This implies that $\|A(\sigma)\|\leq 2\sqrt{\frac{\rho}{\sigma+\rho}}$.
\end{proof}

\subsection{Spectrum}

The spectrum $\mathcal{S}(A(\sigma))$ of the operator $A(\sigma)$ is defined by
$$
\mathcal{S}(A(\sigma)) = \{z \in \R : (z\I-A(\sigma)) \mbox{ is not invertible}\}.
$$
Since  $\|A(\sigma)\|\leq 2\sqrt{\frac{\rho}{\sigma+\rho}}$, we know that $\mathcal{S}(A(\sigma))  \subset \left[ -2\sqrt{\frac{\rho}{\sigma+\rho}}, 2\sqrt{\frac{\rho}{\sigma+\rho}} \right]$. \textcolor{blue}{The spectrum is the support of the spectral measure of the operator $A(\sigma)$.}

\begin{proposition}
The spectral measure of the operator $A(\sigma)$ is purely discrete with atoms at points $s_k^\pm (\sigma)$ defined  for $k=0,1, 2, \ldots$ by
\begin{equation}
    \label{defsk}
    s_k^\pm  (\sigma)= \pm \sqrt{\frac{\rho}{\sigma+\rho+k}};
\end{equation}
the mass at point $s_k^\pm (\sigma)$ is
\begin{equation}
    \label{defrk}
    r_k (\sigma) = \frac{\sigma+\rho}{2(\sigma+\rho+k)} \frac{(\sigma+\rho+k)^k}{k!} e^{-(\sigma+\rho+k)}.
\end{equation}
\end{proposition}

\begin{proof}
Let us consider some vector $P(\sigma; x)$, a priori not necessarily in $H(\sigma)$, such that $A(\sigma)P(\sigma;x)=x P(\sigma;x)$ for some real number $x$. By setting without loss of generality $P_{-1}(\sigma;x)=0$ and $P_0(\sigma;x)=1$, we have for $n \geq 0$
\begin{equation}
\label{recurQ}
\rho  P_{n+1}(\sigma;x) -(n+\sigma+\rho) x P_n(\sigma;x) +n P_{n-1}(\sigma;x)=0.
\end{equation}
This recurrence formula defines an orthogonal polynomial system (OPS), which has been studied by Karlin and McGregor in \cite{KMcG} for $\sigma =0$. The orthogonality is checked by using Favard condition \cite{ismail}. Indeed, the above polynomials satisfy a recurrence relation of the type
$$
P_{n+1}(\sigma;x) =(a_n x+b_n) P_n(\sigma;x) -c_n P_{n-1}(\sigma;x)
$$
with $$
a_n = \frac{n+\sigma+\rho}{\rho}, \; b_n=0, \; \mbox{and} \;  c_n=\frac{n}{\rho}.
$$
for $n\geq 0$ so that Favard condition $a_{n-1}a_n c_n>0$ for $n>1$ is satisfied.

It is easily checked  by using Equation~\eqref{recPC} that the polynomials $P_n(\sigma;x)$ are related to Charlier polynomials as follows: for $n\geq 0$,
\begin{equation}
    \label{PCrel}
  P_n(\sigma;x) = \frac{1}{x^n} C_n\left(\frac{\rho-(\sigma+\rho)x^2}{x^2};\frac{\rho}{x^2}  \right). 
\end{equation}

To determine the orthogonality  measure of the polynomials $P_n(\sigma;x)$, which also defines the spectrum and the spectral measure of the operator $A(\sigma)$, we follow the general method given in \cite{ismail,chihara}.

The polynomials $P_n(\sigma;z)$ are the successive denominators of the continued fraction
$$
\chi(\sigma;z) =  \cfrac{a_0}{a_0z+b_0 -\cfrac{c_1}{a_1 z+b_1 -\cfrac{c_2}{a_2 z +b_2 +\cdots}}}.
$$
\textcolor{blue}{This continued fraction is introduced because the domain where  this function is not defined is precisely the support of the spectral/orthogonality measure, which is then determined by using Perron-Stieltjes inversion formula (see \cite{ismail} for details).}

To obtain an explicit expression for the continued fraction $\chi(\sigma;z)$, we introduce the polynomials of the second kind $P^*_n(\sigma;z)$ associated with the polynomials $P_n(\sigma;z)$ and  satisfying the same recurrence relation~\eqref{recurQ} but with the initial conditions $P^*_0(\sigma;z) =0$ and $P^*_1(\sigma;z)= a_0 = \frac{\sigma+\rho}{\rho}$. It is easily checked that these polynomials are related to the Charlier polynomials of the second kind $C_n^*(z;a)$ as 
$$
P^*_n(\sigma;z) = -\frac{\sigma+\rho}{x^{n+1}} C^*_n\left(\frac{\rho-(\sigma+\rho)x^2}{x^2};\frac{\rho}{x^2}  \right).
$$
(For polynomials of the second kind, indices usually start from 0 and not -1.)


By using Equation~\eqref{PCcf}, we have
\begin{eqnarray*}
\chi(\sigma;x) &= &\lim_{n\to \infty} \frac{P^*_n(\sigma;x)}{P_n(\sigma;x)} = \frac{(\sigma+\rho)}{x\left(\rho+\sigma-\frac{\rho}{x^2}\right)} \Phi\left(1,\rho+\sigma-\frac{\rho}{x^2}+1;-\frac{\rho}{x^2}\right)\\
&=&
 \sum_{m=0}^\infty \frac{(\rho+\sigma)}{x\left(\rho+\sigma-\frac{\rho}{x^2}\right) \ldots\left(\rho+\sigma-\frac{\rho}{x^2}+m\right)  }\left( -\frac{\rho}{x^2}\right)^m.
\end{eqnarray*}

It is obvious that $x=0$ is a removable singularity \textcolor{blue}{and $\chi(\sigma;0)=0$}. The actual poles are the points
$$
s_k^\pm (\sigma) = \pm \sqrt{\frac{\rho}{\sigma+\rho+k}}
$$
for $k= 0, 1, 2, \ldots$. The rational fraction
$$
\frac{(\rho+\sigma)}{x\left(\rho+\sigma-\frac{\rho}{x^2}\right) \ldots\left(\rho+\sigma-\frac{\rho}{x^2}+m\right)  }
$$
has a pole at point $s_k^\pm (\sigma)$ for $k \leq m$. For $m = k+\ell$, we have
$$
\left.\prod_{j=0, j\neq k}^m \frac{1}{\rho+\sigma +j -\frac{\rho}{x^2}} \right|_{x=s_k^\pm} = \frac{(-1)^k}{k! \ell !}
$$
The residue at pole $s_k^\pm (\sigma)$ of the function
$$
\frac{1}{x\left(\rho+\sigma-\frac{\rho}{x^2}  \right)}
$$
is $\frac{1}{2(\sigma+\rho+k)}$. It follows that the residue of the function $\chi(\sigma;x)$ at pole $s_k^\pm (\sigma)$ is
$$
r_k (\sigma) = \frac{\sigma+\rho}{2(\sigma+\rho+k)} \frac{(\sigma+\rho+k)^k}{k!} e^{-(\sigma+\rho+k)}.
$$
We deduce that the orthogonality measure associated with the polynomials $P_n(\sigma;x)$, which is also the spectral measure of the operator $A(\sigma)$, is purely discrete with atoms at points $s_k^\pm (\sigma)$ for $k =0,1,2, \ldots$ and with mass $r_k (\sigma)$ at $s_k^\pm (\sigma)$. The vectors $P(\sigma;s^\pm_k(\sigma))$, $k=0,1, \ldots$ are eigenvectors of the operator $A(\sigma)$.
\end{proof}

From the above proposition, the spectral measure $d\psi(\sigma;x)$ is given by
\begin{equation}
\label{defpsi}
d\psi(\sigma;x) = \sum_{k=0}^\infty r_k(\sigma) (\delta_{s_k^+(\sigma)}(dx)+\delta_{s_k^-(\sigma)}(dx)),
\end{equation}
where $\delta_a(dx)$ is the Dirac mass at point $a$. The polynomials $P_n(\sigma;x)$ satisfy the orthogonality relations \cite{ismail}
\begin{equation}
    \label{orthrel}
    \int_{-\infty}^\infty P_n(\sigma;x)P_m(\sigma;x) d\psi(\sigma;x) = \frac{1}{\pi_n(\sigma)} \delta_{n,m},
\end{equation}
where $\pi_n(\sigma)$ is defined by Equation~\eqref{defpin}. {\color{blue} The continued fraction $\chi(\sigma;z)$ is such that 
$$
\chi(\sigma;z) = \int_{-\infty}^\infty \frac{1}{z-x} d\psi(\sigma;x)
$$
and by using the arguments in \cite{ismail} it is possible to obtain the following relation for $z$ not in the support of the measure $d\psi(\sigma;x)$:
\begin{equation}
\label{eqconnexion}
\int_{-\infty}^\infty \frac{P_m(\sigma;x) }{z-x}d \psi(\sigma;x)  = P_m(\sigma;z) \chi(\sigma;z) - P^*_m(\sigma;z).
 \end{equation}

Finally, it is worth noting that the exponential generating function of the polynomials  $P_n(\sigma;x)$ is given by
\begin{equation}
    \label{expogenPn}
    \sum_{n=0}^\infty P_n(\sigma;x) \frac{z^n}{n!} = \mathcal{C}\left(\frac{\rho-(\sigma+\rho)x^2}{x^2};\frac{\rho}{x^2};\frac{z}{x}  \right) = e^{\frac{z}{x}}\left(1 - \frac{z x}{\rho}   \right)^{\frac{\rho}{x^2}-\sigma-\rho},
\end{equation}
}
where we have used the definition of $\mathcal{C}(x;a;z)$ given  by Equation~\eqref{PCgen}.

\section{Transient characteristics of the $M/M/\infty$ queue}
\label{transient}

In this section, we use the spectral properties of the operator $A(\sigma)$ to compute the probability mass functions of the transient characteristics $\kappa(\sigma)$, $\nu(\sigma)$, and $\mathfrak{d}(\sigma)$. In a first step, we consider the random variable $\nu(\sigma)$, which is related to the number of customers in the queue at time $t$.

\begin{proposition}
Under the assumption that the system is empty at the time origin, the probability mass function of the random variable $\nu(\sigma)$ equal to the number of customers in the system upon departure of the observer is given by
\begin{equation}
    \label{massnu}
    \P(\nu(\sigma)=m) = \frac{\sigma}{\sigma+\rho} \frac{\rho^m}{m!}  \int_{-\infty}^\infty \frac{P_m(\sigma;x) }{1-x}d \psi(\sigma;x),
\end{equation}
which can be rewritten as
\begin{equation}
\label{massnubis}
\P(\nu(\sigma)=m ) = \sigma  \frac{\rho^m}{m!} \int_0^\infty \frac{C_m(x;\rho)}{\sigma+x}d\mathcal{P}_\rho(x),
\end{equation}
where $d\mathcal{P}_\rho(x)$ is the Poisson measure on $\mathbbm{R}$ with mean $\rho$.
\end{proposition}

\begin{proof}
By using Equation~\eqref{eqnu}, we know that 
$$
\P(\nu(\sigma)=m ) = \frac{\sigma}{m+\sigma+\rho}{}^te_0(\I-A(\sigma))^{-1}e_m .
$$
Since for arbitrary $m$
$$
e_m = {\pi_m(\sigma)}  \int_{-\infty}^\infty {P_m(\sigma;x)}P(\sigma;x)d \psi(\sigma;x),
$$
$P(\sigma;x)$ denoting the vector with components $P_\ell(\sigma;x)$ for $\ell \geq 0$, we have
$$
(\I-A(\sigma))^{-1}e_m= {\pi_m(\sigma)}  \int_{-\infty}^\infty \frac{P_m(\sigma;x) P(\sigma;x)}{1-x}d \psi(\sigma;x)
$$
by the spectral identity \cite{reed}. Equation~\eqref{massnu} easily follows by using the fact that $P_0(\sigma;x) =1$ for all $x$.

 For the particular value $z=1$, we have $$P_m(\sigma;1)=C_m(-\sigma;\rho) \mbox{ and } P^*_m(\sigma;1)=-(\sigma+\rho)C^*_m(-\sigma;\rho),
 $$
and in addition
$$
\chi(\sigma;1) = \frac{\sigma+\rho}{\sigma} \Phi(1;\sigma+1;-\rho)=(\sigma+\rho) \chi_0(\sigma;\rho),
$$
so that by using Equation~\eqref{eqconnexion}
{\color{blue}
\begin{align*}
 \int_{-\infty}^\infty \frac{P_m(\sigma;x) }{1-x}d \psi(\sigma;x) &=(\sigma+\rho) \left({C_m(-\sigma,\rho)} \chi_0(\sigma;\rho)+ C^*_m(-\sigma;\rho)\right) \\
 &= (\sigma+\rho) \int_0^\infty \frac{C_m(x;\rho)}{\sigma+x}d\mathcal{P}_\rho(x)
\end{align*}
where we have used Equation~\eqref{henrici_charlier} in the last step.}  Equation~\eqref{massnubis} then follows.
\end{proof}

\begin{corollary}[\cite{KMcG}]
Under the assumption that the system is empty at the time origin, the probability mass function of the random variable $N(t)$ equal to the number of customers in the system at time $t$ is given by Equation~\eqref{distibN}.
\end{corollary}

\begin{proof}
From Equation~\eqref{massnubis} and by Laplace inversion, we have 
\begin{eqnarray*}
\P(N(t)=m~|~N(0)=0)  &=&  \frac{\rho^m}{m!}  \int_0^\infty {C_m(x;\rho)}e^{-xt}d\mathcal{P}_\rho(x) \\ &=&  \frac{\rho^m}{m!} \sum_{n=0}^\infty {C_m(n;\rho)}e^{-nt} \frac{\rho^n}{n!} e^{-\rho} \\
&=&  \frac{\rho^m}{m!} \sum_{n=0}^\infty {C_n(m;\rho)}e^{-nt} \frac{\rho^n}{n!} e^{-\rho} \\
&=& e^{-\rho\left(1- e^{-t}\right)} \frac{(\rho (1-e^{-t}))^m}{m!},
\end{eqnarray*}
where we have used the symmetry relation~\eqref{symCharlier}  and the generating function of the Charlier polynomials given by~\eqref{PCgen}.
This completes the proof.
\end{proof}

We now consider the random variable $\kappa(\sigma)$, whose probability generating function seems to be unknown in the technical literature.

\begin{proposition}
The \textcolor{blue}{generating function} of the random variable $\kappa(\sigma)$, which is the number of arrivals or departures in an initially empty $M/M/\infty$ system during an exponential period of time with mean $1/\sigma$,  is given by
{\color{blue}\begin{equation}
    \label{pmfkappa}
   \E\left(z^{\kappa(\sigma)}\right) = \sigma z e^{\rho z(1-z)}   \sum_{n=0}^\infty  \frac{1}{\sigma+\rho+n-\rho z^2}    \frac{(-\rho z(1-z))^n}{n!}.
\end{equation}}
\end{proposition}

\begin{proof}
From Equation~\eqref{kappamu}, we have
\begin{eqnarray*}
\P(\kappa(\sigma)=k) &=& \sum_{m=0}^\infty \P(\kappa(\sigma)=k, \nu(\sigma)=m)   \\
&=&\sum_{m=0}^\infty \frac{\sigma}{m+\sigma+\rho} {}^te_0 A(\sigma)^{k-1} e_m \nonumber \\
&=& \frac{\sigma}{\sigma+\rho}\sum_{m=0}^\infty \frac{\rho^m}{m!} \int_{-\infty}^\infty x^{k-1}P_m(\sigma;x)d\psi(\sigma;x)  .
\end{eqnarray*}
{\color{blue}
It follows that the generating function of $\kappa(\sigma)$ is given for $|z|<1$ by
$$
\E\left(z^{\kappa(\sigma)}\right) = \frac{\sigma }{\sigma+\rho}\sum_{m=0}^\infty \frac{\rho^m}{m!} \int_{-\infty}^\infty \frac{z P_m(\sigma;x)}{1- z x}d\psi(\sigma;x)  .
$$

By using Equation~\eqref{eqconnexion} and setting $Z= {\sigma+\rho}-\rho  z^2$, we have for real $z\neq 0$
\begin{align*}
 \int_{-\infty}^\infty \frac{ z P_m(\sigma;x)}{1- z x}d\psi(\sigma;x) &= \int_{-\infty}^\infty \frac{  P_m(\sigma;x)}{\frac{1}{z}-  x}d\psi(\sigma;x) \\
&= P_m\left(\sigma;\frac{1}{z}\right) \chi\left(\sigma;\frac{1}{z}\right) - P^*_m\left(\sigma;\frac{1}{z}\right) \\
&= (\sigma+\rho) z^{m+1} \left( C_m(-Z;\rho z^2) \chi_0\left(Z;\rho z^2  \right) + C^*_m(-Z;\rho z^2)\right) \\
&= (\sigma+\rho) z^{m+1} \int_0^\infty \frac{C_m(x;\rho z^2)}{Z+x} d\mathcal{P}_{\rho z^2} (x).
\end{align*}
It follows that
$$
\E\left(z^{\kappa(\sigma)}\right) = \sigma z e^{\rho z} \int_0^\infty \frac{\left(1-\frac{1}{z}  \right)^x}{\sigma+\rho+x-\rho z^2} d\mathcal{P}_{\rho z^2} (x)
$$
and Equation~\eqref{pmfkappa} follows.
}
\end{proof}
 
 {\color{blue} Let $K(t)$ be the number of arrivals and departures in the system up to time $t$. By definition, we have
$$
\P(\kappa(\sigma)=k ) =\int_0^\infty \P(K(t) = k) \sigma e^{-\sigma t} dt.
$$
We then have the following result.

\begin{corollary}
    The generating function of the number $K(t)$ of arrivals and departures in the system up to time $t$ is given by
    \begin{equation}
        \label{funcK}
          \E\left(z^{K(t)}\right) = z e^{ \rho z (1-z)(1-e^{-t}) -\rho(1-z^2)t} .
    \end{equation}
\end{corollary}

\begin{proof}
    By Laplace inversion, we have from Equation~\eqref{pmfkappa}
    $$
     \E\left(z^{K(t)}\right)= z e^{\rho z(1-z)} \sum_{n=0}^\infty \frac{(-\rho z (1-z))^n}{n!}  e^{-(\rho+n-\rho z^2)t}
    $$
    and Equation~\eqref{funcK} follows.
\end{proof}

It is easily checked that the first moment of $K(t)$ is given by 
$$
\E(K(t)) = 1-\left(1-e^{-t}\right) \rho +2 t \rho
$$
and the second moment by
$$
E(K(t)^2) = e^{-2 t} \rho  \left(\rho +e^t \left(4+(-2+4 t) \rho +e^t (-4+6 t+\rho +4 (-1+t) t \rho )\right)\right).
$$
By taking Laplace transforms, we obtain
\begin{equation}
    \label{Ekappa}
    \E(\kappa(\sigma)) = 1+\frac{\rho  (2+\sigma )}{\sigma  (1+\sigma )}
\end{equation}
and
$$
 \E(\kappa(\sigma)^2) =   1+ \frac{\rho (8+3 \sigma )}{\sigma(1+\sigma)}+\frac{2 \rho ^2 (8+\sigma  (16+\sigma  (7+\sigma )))}{\sigma ^2 (1+\sigma )^2 (2+\sigma )}. \label{Ekappa2}
$$
}

 For the distribution of $\mathfrak{d}(\sigma)$, we use the same technique and we obtain the following result.
 
 \begin{proposition}
 \label{gend}
 The \textcolor{blue}{generating}  function of the random variable $\mathfrak{d}(\sigma)$, equal to the number of departures from the initially empty $M/M/\infty$ system in an exponentially distributed time frame with mean $1/\sigma$, is given by
{\color{blue} \begin{eqnarray}
     \label{massd}
      \E\left(z^{\mathfrak{d}(\sigma)}\right) &=& \sigma \sum_{n=0}^\infty\frac{(\sigma +n)^n e^{-(\sigma + n)}}{n!} \frac{1}{\sigma+\rho+n- \rho z} \\
      &=& \frac{\sigma}{\sigma +\rho(1-z)}\Phi(1,\sigma+\rho(1-z)+1;\rho(1-z)) . \label{massdbis}
\end{eqnarray}
}
 \end{proposition}
 
 \begin{proof}
 By using Equation~\eqref{conservationbis}, we have
{\color{blue} \begin{eqnarray*}
 \P(\mathfrak{d}(\sigma)=k) &=& \P(\kappa(\sigma)=2k+1+\nu(\sigma)) \\
&=& \sum_{m=0}^\infty \P(\kappa(\sigma)=2k+m+1,\nu(\sigma)=m)\\
&=& \sum_{m=0}^\infty \frac{\sigma}{m+\sigma+\rho} {}^te_0 A(\sigma)^{2k+m} e_m\\
&=&\frac{\sigma}{\sigma+\rho}\sum_{m=0}^\infty \frac{\rho^m}{m!} \int_{-\infty}^\infty x^{2k+m}P_m(\sigma;x)d\psi(\sigma;x)\\
&=&\frac{\sigma}{\sigma+\rho} \int_{-\infty}^\infty x^{2k}\mathcal{C}\left(\frac{\rho-(\sigma+\rho)x^2}{x^2};\frac{\rho}{x^2};\rho  \right)d\psi(\sigma;x),
\end{eqnarray*}
where we have used the exponential generating function of the polynomials $P_n(\sigma;x)$ given by Equation~\eqref{expogenPn}. Hence,
$$
\P(\mathfrak{d}(\sigma)=k) =  \frac{\sigma e^{{\rho}}}{\sigma+\rho}\int_{-\infty}^\infty x^{2k} (1-x^2)^{\frac{\rho-(\sigma+\rho) x^2}{x^2}} d\psi(\sigma;x)
$$
and  Equation~\eqref{massd} is obtained by using the definition of the measure $d\psi(\sigma;x)$ given by Proposition~\eqref{defpsi}.

To obtain Equation~\eqref{massdbis}, let us consider the function
$$
\frac{1}{\sigma -X}\Phi(1,\sigma-X+1;-X) = \sum_{m=0}^\infty \frac{(-X)^m}{\prod_{\ell=0}^m(\sigma+\ell-X)},
$$
where $\Phi(\alpha,\beta;z)$ is the Kummer function. The above function has poles at point $\sigma+n$ for $n\geq 0$. At point $\sigma+n$, the residue is equal to 
$$
\frac{(\sigma+n)^n}{n!}e^{-(\sigma+n)}.
$$
It follows that
$$    
\frac{1}{\sigma -X}\Phi(1,\sigma-X+1;-X) = \sum_{n=0}^\infty  \frac{(\sigma+n)^n}{n!}e^{-(\sigma+n)}  \frac{1}{\sigma+n-X}
$$
and Equation~\eqref{massdbis} follows by taking $X=-\rho(1-z)$.
}
 \end{proof}

{\color{blue} It is worth noting that taking $z=1$ in Equation~\eqref{massd} reads
$$
\sum_{n=0}^\infty\frac{(\sigma +n)^{n-1} e^{-(\sigma + n)}}{n!} =1,
$$
which is precisely Euler's formula \cite{polya}
$$
e^{\alpha z} = \alpha \sum_{n=0}^\infty \frac{(\alpha+n)^n}{n!}(z e^{-z})^n
$$
for $z=1$ and $\alpha=\sigma$.

As a consequence of Proposition~\eqref{gend}, we can state the following result.

\begin{proposition}
    The random variable $D(t)$ is Poisson with mean $\rho(e^{-t}-1+t)$.
\end{proposition}

\begin{proof}
 By using the integral representation of Kummer function \cite{Abramowitz}, we have
 \begin{eqnarray*}
    \E\left(z^{\mathfrak{d}(\sigma)}\right) &=& \sigma \int_0^1 e^{z(1-\rho)u}(1-u)^{\sigma+\rho(1-z)-1}du \\ &=& \sigma e^{\rho(1-z)}\int_0^1 e^{-z(1-\rho)u}u^{\sigma+\rho(1-z)-1}du
 \end{eqnarray*}
 and via the variable change $u=e^{-t}$, we have
$$
  \E\left(z^{\mathfrak{d}(\sigma)}\right) = \sigma e^{\rho(1-z)} \int_0^\infty  e^{-z(1-\rho)e^{-t}-(\sigma+\rho(1-z))t}dt
  $$
  and then since 
  $$
\int_0^\infty  \E\left(z^{D(t)}\right) \sigma e^{-\sigma t} dt =  \E\left(z^{\mathfrak{d}(\sigma)}\right)
$$
  we have by Laplace inversion
  $$
   \E\left(z^{D(t)}\right) = e^{\rho(1-z) (1-e^{-t}-t)},
   $$
which is the generating function of Poisson random variable with mean $\rho(e^{-t}-1+t)$.   
\end{proof}

We thus have proved that we can recover the result for the random variable $KD(t)$ obtained by probabilistic arguments via spectral theory. Note that the generating function of $K(t)$ seems to be unknown in the queuing literature. In the next section, we investigate how the results apply for an $M/M/\mathfrak{c}/\mathfrak{c}$ queue where $\mathfrak{c}$is some positive integer.

\section{Transient characteristics of the $M/M/\mathfrak{c}/\mathfrak{c}$ queue}
\label{mmcc}

In the case of an $M/M/\mathfrak{c}/\mathfrak{c}$ queue, we consider as in the previous sections an observer joining an initially empty queue and staying in the system for an exponentially distributed random time with mean $1/\sigma$ with $\sigma>0$.

\subsection{Notation}

Let us introduce the discrete-time process $\left(\mathfrak{n}^{[\mathfrak{c}]}_k(\sigma)\right)$ describing the number of customers in the $M/M/\mathfrak{c}/\mathfrak{c}$ queue without taking into account the observer; $\mathfrak{n}^{[\mathfrak{c}]}_k(\sigma)$ is the number of customers in the queue at the $k$th event corresponding either to a customer arrival, or a service completion or the departure of the observer from the system. When the observer leaves the system, the process $\left(\mathfrak{n}^{[\mathfrak{c}]}_k(\sigma)\right)$ is absorbed at state $-1$. 

The state space of the  discrete-time Markov chain $\left(\mathfrak{n}^{[\mathfrak{c}]}_k(\sigma)\right)$ is $\{-1,0, 1,2, \ldots, \cc\}$  with transition matrix $\mathcal{A}^{[\mathfrak{c}]}(\sigma)$ given by
$$
\mathcal{A}^{[\mathfrak{c}]}(\sigma)= \begin{pmatrix}
1 & 0 & 0 & 0 & 0 & \ldots \\
\frac{\sigma}{\sigma+\rho} & 0 & \frac{\rho}{\sigma+\rho} & 0 & 0 & \ldots \\
\frac{\sigma}{1+\rho+\sigma} &\frac{1}{1+\rho+\sigma} & 0 & \frac{\rho}{\rho+\sigma+1}  & 0 & \ldots \\
\frac{\sigma}{2+\rho+\sigma} & 0 &\frac{2}{2+\rho+\sigma}  & 0 &   & \ldots \\
\vdots & \vdots &  & &  & \frac{\rho}{\sigma+\mathfrak{c}-1+\rho}\\
\frac{\sigma}{\sigma+\mathfrak{c}} & \vdots &  & & \frac{\mathfrak{c}}{\mathfrak{c}+\sigma} & 0
\end{pmatrix}.
$$
The non-zero coefficients of the matrix $\mathcal{A}^{[\mathfrak{c}]}(\sigma)$ are given by $\mathcal{A}^{[\mathfrak{c}]}_{-1,-1}(\sigma)=1$ (the state $-1$ being absorbing) and for $0 \leq n < \mathfrak{c}$
$$
\mathcal{A}^{[\mathfrak{c}]} _{n,-1}(\sigma)= \frac{\sigma}{n+\sigma+ \rho},\,  \mathcal{A}^{[\mathfrak{c}]}_{n,n-1}(\sigma)= \frac{n}{n+\sigma+\rho}, \mbox{ and }\mathcal{A}^{[\mathfrak{c}]}_{n,n+1}(\sigma) = \frac{\rho}{n+\sigma+\rho}
$$
along with
$$
\mathcal{A}^{[\mathfrak{c}]}_{\mathfrak{c},-1}(\sigma)= \frac{\sigma}{\sigma+ \mathfrak{c}}, \mbox{ and }\mathcal{A}^{[\mathfrak{c}]}_{\mathfrak{c},\mathfrak{c}-1}(\sigma) = \frac{\mathfrak{c}}{\mathfrak{c}+\sigma}.
$$

The sub-matrix $A^{[\mathfrak{c}]}(\sigma)$ 
obtained by deleting the first row and the first column of matrix $\mathcal{A}^{[\mathfrak{c}]}(\sigma)$ is a tridiagonal matrix with non-zero coefficients given for   $0\leq n < \mathfrak{c} $
$$
A^{[\mathfrak{c}]}_{n,n+1}(\sigma)= \frac{\rho}{n+\sigma+\rho}, \quad A^{[\mathfrak{c}]}_{n,n-1}(\sigma)= \frac{n}{n+\sigma+\rho}
$$
together with
$$
A^{[\mathfrak{c}]}_{\mathfrak{c},\mathfrak{c}-1}(\sigma)= \frac{\mathfrak{c}}{\mathfrak{c}+\sigma}
$$
with the convention $A^{[\mathfrak{c}]}_{-1,0}(\sigma)=0$.  (We let the indices of the coefficients of matrix $A^{[\mathfrak{c}]}(\sigma)$ range from 0 to $\cc$ as it is more convenient for recurrence relations appearing in the analysis.) The $(\mathfrak{c}+1)\times (\mathfrak{c}+1)$ matrix $A^{[\mathfrak{c}]}(\sigma)$ is sub-stochastic and describes the transition probabilities of the Markov chain $(\mathfrak{n}_k^{[\mathfrak{c}]}(\sigma))$ before absorption.

\subsection{Spectral properties}

Let $H^{[\mathfrak{c}]}(\sigma) \subset H(\sigma)$ be the vector space such that the components of a vector $f \in H^{[\mathfrak{c}]}(\sigma)$ are zero for indices larger than $\mathfrak{c}$. The matrix $A^{[\mathfrak{c}]}(\sigma)$ is not selfadjoint  but can nevertheless be diagonalized. An eigenvalue of $A^{[\mathfrak{c}]}(\sigma)$ is such that there exists a vector $f(x) \in H^{[\mathfrak{c}]}(\sigma)$ satisfying the recursion
$$
\rho f_{n+1}(x) -(\rho +n +\sigma)x f_n(x) + nf_{n-1}(x) =0
$$
for $0\leq n <\mathfrak{c}$ (with the convention $f_{-1}(x)=0$) and
\begin{equation}
    \label{condlimit}
     -(\mathfrak{c}+\sigma) x f_\mathfrak{c}(x) + \mathfrak{c} f_{\mathfrak{c}-1}(x) =0.
\end{equation}
Without loss of generality, we can set $f_0(x)=1$ and $f_n(x)$ then satisfies the same recursion as $P_n(\sigma;x)$ for $n=-1,0,1 \ldots, \mathfrak{c}$. The limiting condition~\eqref{condlimit} implies the point $x$ is an eigenvalue only if
$$
 -(\mathfrak{c}+\sigma) x f_\mathfrak{c}(x) + \mathfrak{c} f_{\mathfrak{c}-1}(x) =0 \Longleftrightarrow P_{\mathfrak{c}+1}(\sigma;x) = xP_\mathfrak{c}(\sigma;x).
$$

For $\sigma>0$, the moment functional associated with the measure $d\psi(\sigma;x)$  is positive-definite on the set of atoms $\{s_k^\pm(\sigma), k=0,1, \ldots\}\subset[-1,1]$ since the mass at each atom is positive. From the theory of orthogonal polynomials \cite[Theorem~5.2]{chihara}, the zeros of $P_n(\sigma;x)$, $n\geq 1$ are real, simple and  located in the interior of $[-1,1]$. In addtion, the roots of $P_n(\sigma;x)$ and $P_{n+1}(\sigma;x)$ are interleaved. We then easily deduce via geometric arguments that the equation $P_{\mathfrak{c}+1}(\sigma;x) = xP_\mathfrak{c}(\sigma;x)$ has $\mathfrak{c}+1$ real and simple solutions, denoted by $\xi_{\mathfrak{c},k}(\sigma)$ for $k=0, \ldots,\mathfrak{c}$.

Let $\mathbf{P}_k(\sigma)$ for $k=0,\ldots,\mathfrak{c}$ denote the column vector whose $n$th entry is equal to $P_n(\sigma;\xi_{\mathfrak{c},k}(\sigma))$ for $n=0,\ldots,\mathfrak{c}$. The vectors $\mathbf{P}_k(\sigma)$ for $k=0,\ldots,\mathfrak{c}$ form an orthogonal basis of the space $H^{[\mathfrak{c}]}(\sigma)$. Moreover, let  $\mathbf{P}^*_k(\sigma)$ for $k=0,\ldots,\mathfrak{c}$ denote the column vector whose $n$th entry is equal to $P^*_n(\sigma;\xi_{\mathfrak{c},k}(\sigma))$ for $n=0,\ldots,\mathfrak{c}$.

By using the orthogonality of the vectors $\mathbf{P}_k(\sigma)$ for $k=0,\ldots,\mathfrak{c}$, we have on the one hand
$$
(e_0,(z \mathbbm{I}- A^{\cc}(\sigma))^{-1}e_0) = \sum_{k=0}^\cc \frac{1}{\| \mathbf{P}_k(\sigma)\|^2}\frac{1}{z -\xi_{\cc,k}(\sigma)}.
$$
On the other hand, we have for any constant $\gamma$ and fixed $z$ not in the set of the roots $\{\xi_{\cc,k}(\sigma), k=0, \ldots, \cc\}$
\begin{multline*}
(z \mathbbm{I}- A^{[\cc]}(\sigma))(-\mathbf{P}^*(\sigma;z) +\gamma \mathbf{P}(\sigma;z)) = e_0 - \\ \frac{\rho}{\cc+\rho}(P^*_{\cc+1}(\sigma;z)-z P^*_{\cc}(\sigma;z) -\gamma (P_{\cc+1}(\sigma;z)-z P_{\cc}(\sigma;z)))e_{\cc},
\end{multline*}
where $\mathbf{P}(\sigma;z)$ (resp. $\mathbf{P}^*(\sigma;z)$) is the vector with entries ${P}_n(\sigma;z)$ (resp. ${P}^*_n(\sigma;z)$) for $n=0,\ldots,\cc$. By choosing
$$
\gamma=  \frac{P^*_{\cc+1}(\sigma;z)-z P^*_{\cc}(\sigma;z)}{P_{\cc+1}(\sigma;z)-z P_{\cc}(\sigma;z)},
$$
we have
$$
(z \mathbbm{I}- A^{\cc}(\sigma))^{-1}e_0 =-\mathbf{P}^*(\sigma;z) + \frac{P^*_{\cc+1}(\sigma;z)-z P^*_{\cc}(\sigma;z)}{P_{\cc+1}(\sigma;z)-z P_{\cc}(\sigma;z)}\mathbf{P}(\sigma;z).
$$
We then deduce that
$$
(e_0,(z \mathbbm{I}- A^{\cc}(\sigma))^{-1}e_0)=   \frac{P^*_{\cc+1}(\sigma;z)-z P^*_{\cc}(\sigma;z)}{P_{\cc+1}(\sigma;z)-z P_{\cc}(\sigma;z)}.
$$
and hence,
$$
\mathfrak{m}^{[\cc]}_k (\sigma)= \frac{1}{\| \mathbf{P}_k(\sigma)\|^2} =  \frac{P^*_{\cc+1}(\sigma;\xi_{\cc,k}(\sigma))-\xi_{\cc,k}(\sigma) P^*_{\cc}(\sigma;\xi_{\cc,k}(\sigma))}{P'_{\cc+1}(\sigma;\xi_{\cc,k}(\sigma))- \xi_{\cc,k}(\sigma)P'_{\cc}(\sigma;\xi_{\cc,k}(\sigma)) -P'_{\cc}(\sigma;\xi_{\cc,k}(\sigma)) }.
$$

Let us introduce the discrete measure $d\psi^{[\cc]} (\sigma;x)$, which has an atom at point $\xi_{\cc,k}(\sigma)$ with mass $\mathfrak{m}^{[\cc]}_k(\sigma)$ for $k=0, \ldots, \cc$. We have
$$
 \int_{-\infty}^\infty \frac{1}{z-x}d\psi^{[\cc]} (\sigma;x) =  \frac{P^*_{\cc+1}(\sigma;z)-z P^*_{\cc}(\sigma;z)}{P_{\cc+1}(\sigma;z)-z P_{\cc}(\sigma;z)}.
$$
By computing $(e_m,(z \mathbbm{I}- A^{\cc}(\sigma))^{-1}e_0)$, we obtain
\begin{equation}
    \label{eqtechcc1}
    P_m(\sigma;z)  \frac{P^*_{\cc+1}(\sigma;z)-z P^*_{\cc}(\sigma;z)}{P_{\cc+1}(\sigma;z)-z P_{\cc}(\sigma;z)} - P^*_m(\sigma;z)= \int_{-\infty}^\infty \frac{P_m(\sigma;x)}{z-x}d\psi^{[\cc]} (\sigma;x).
\end{equation}

To conclude this section, let us consider the matrix $B^{[\cc]}$ defined by
$$
\mathcal{B}^{[\mathfrak{c}]}= \begin{pmatrix}
  -\rho  & \rho & 0 & 0 & \ldots \\
 1 & -(1+\rho) & \rho & 0 & \ldots \\
 0 &2  & -(2+\rho) & \rho &  \ldots \\
 \vdots &  & \cc & -(\rho+\cc) & \rho \\
  &  & & \cc & -\cc
\end{pmatrix}.
$$
The matrix $\mathcal{B}^{[\cc]}$ is the infinitesimal generator of the Markov process $(\mathfrak{n}^{[\cc]}(t))$ describing the number of customers in the $M/M/\mathfrak{c}/\mathfrak{c}$ system. This matrix induces a selfadjoint operator in the Hilbert space $H$
defined by
$$
H=\left\{f \in \mathbb{R}^\mathbb{N}: \sum_{n=0}^\infty f_n^2 \frac{\rho^n}{n!}<\infty   \right\}
$$
(see \cite{KMcG} for details) and   can be diagonalized by using the same technique as above. The eigenvalues satisfy the equation
$$
C_{\cc+1}(-x;\rho) = C_\cc(-x,\rho),
$$
which has $\cc+1$ non-positive solutions, denoted by $-\sigma_{\cc,k}$ for $k=0, \ldots, \cc$. Note that 0 is an eigenvalue associated with eigenvector $e^{[\cc]}$ with all components equal to 1. Introducing the measure $d\phi^{[\cc]}(\rho;x)$ with atoms at point $\sigma_{\cc,k}$ with mass
$$
m^{[\cc]}_k= \frac{C^*_{\cc+1}(\sigma_k;\rho) - C^*_{\cc}(\sigma_k;\rho)}{C'_{\cc+1}(\sigma_k;\rho)- C'_{\cc}(\sigma_k;\rho)}
$$
for $k=0,\ldots,\cc$, we have
$$
\int_{0}^\infty \frac{1}{z+x} d\phi^{[\cc]}(\rho;x) =  -\frac{C^*_{\cc+1}(-z;\rho) - C^*_{\cc}(-z;\rho)}{C_{\cc+1}(-z;\rho)- C_{\cc}(-z;\rho)}
$$
Finally, we have the relation for $m=0,\ldots, \cc$
\begin{equation}
    \label{eqtechcc2}
    C^*_m(-z;\rho)-\frac{C^*_{\cc+1}(-z;\rho) - C^*_{\cc}(-z;\rho)}{C_{\cc+1}(-z;\rho)- C_{\cc}(-z;\rho)} C_m(-z;\rho) = \int_{0}^\infty \frac{C_m(x;\rho)}{z+x} d\phi^{[\cc]}(\rho;x).
\end{equation}

\subsection{Transient characteristics}
As in Section~\ref{transient}, we consider the number  $\nu^{[\cc]}(\sigma)$ of customers in the system when the observer leaves the system. 

\begin{proposition}
The probability mass function of the random variable $\nu^{[\cc]}(\sigma)$ is given by
 \begin{equation}
   \P(\nu^{[\cc]} (\sigma)=m ) = \sigma  \frac{\rho^m}{m!} \int_{0}^\infty \frac{C_m(x;\rho)}{\sigma+x} d\phi^{[\cc]}(\rho;x).
\end{equation}   
\end{proposition}
\begin{proof}
As in  Section~\ref{transient}, we have
\begin{eqnarray*}
\P(\nu^{[\cc]} (\sigma)=m ) &=& \frac{\sigma}{m+\sigma+\rho}{}^te_0(\I-A^{[\cc]}(\sigma))^{-1}e_m \\
&=& \frac{\sigma}{\sigma + \rho} \frac{\rho^m}{m!} \int_{-\infty}^\infty \frac{P_m(\sigma;x)}{1-x}d\psi^{[\cc]} (\sigma;x),
\end{eqnarray*}
where we have used the fact that
$$
e_m = \pi_m \sum_{k=0}^\cc \mathfrak{m}^{[\cc]}_k P_m(\sigma;\xi_{\cc,k}(\sigma))\mathbf{P}_k(\sigma)
$$
and then 
$$
(\I-A^{[\cc]}(\sigma))^{-1}e_m= \pi_m \sum_{k=0}^\cc \mathfrak{m}^{[\cc]}_k \frac{P_m(\sigma;\xi_{\cc,k}(\sigma))}{1-\xi_{\cc,k}(\sigma)}\mathbf{P}_k(\sigma).
$$
It follows that by using Equation~\eqref{eqtechcc1}
\begin{eqnarray*}
   \P(\nu^{[\cc]} (\sigma)=m ) &=&  \frac{\sigma}{\sigma + \rho} \frac{\rho^m}{m!}\left(P_m(\sigma;1)  \frac{P^*_{\cc+1}(\sigma;1)-z P^*_{\cc}(\sigma;1)}{P_{\cc+1}(\sigma;1)- P_{\cc}(\sigma;1))} - P^*_m(\sigma;1) \right) \\
   &=& \sigma  \frac{\rho^m}{m!} \left( - \frac{C^*_{\cc+1}(-\sigma;\rho) -  C^*_{\cc}(-\sigma;\rho)    }{C_{\cc+1}(-\sigma;\rho) - C_{\cc}(-\sigma;\rho) } C_m(-\sigma;\rho) + C^*_m(-\sigma;\rho) \right), \\
   &=& \sigma  \frac{\rho^m}{m!} \int_{0}^\infty \frac{C_m(x;\rho)}{\sigma+x} d\phi^{[\cc]}(\rho;x),
\end{eqnarray*}
where we have used the connection between the polynomials $P_n(\sigma;x)$ and Charlier polynomials, and Equation~\eqref{eqtechcc2} in the last step.
\end{proof}

By Laplace inversion, we have by letting $N^{[\cc]}(t)$ denote the number of customers in the  $M/M/\cc/\cc$ queue at time $t$
$$
  \P(N^{[\cc]} (t)=m )=  \frac{\rho^m}{m!} \int_0^\infty {C_m(x;\rho)} e^{-xt}d\phi^{[\cc]} (\rho;x),
$$
which is the Karlin-McGregor result for the  birth and death $(N^{[\cc]} (t))$ issued from state 0 (see \cite{KMcGBDP} for details).

For the variable $\kappa^{[\cc]}(\sigma)$ equal to the number of customers entering the system or leaving the system, we have the following result.

\begin{proposition}
    The generating function of the random variable $\kappa^{[\cc]}(\sigma)$ is given by
    \begin{equation}
        \label{genkappacc}
        \E\left( z^{\kappa^{[\cc]}(\sigma)}  \right) = \sigma z \int_{0}^\infty \frac{\mathcal{C}^{[\cc]} (x;\rho z^2;\rho z)}{\sigma+\rho-\rho z^2 x}d\phi^{[\cc]}(\rho z^2;x) ,
    \end{equation}
 where
 \begin{equation}
     \label{defCcc}
     \mathcal{C}^{[\cc]} (x;a;z) = \sum_{m=0}^\cc C_m(x;a) \frac{z^m}{m!}.
 \end{equation}
 \end{proposition}

\begin{proof}
    We have
    \begin{eqnarray*}
   \P(\kappa^{[\cc]}(\sigma)=k) &=& \sum_{m=0}^\cc \frac{\sigma}{m+\sigma+\rho} {}^te_0 \left(A^{[\cc]}(\sigma)\right)^{k-1} e_m\\
 &=& \frac{\sigma}{\sigma+\rho}\sum_{m=0}^\cc \frac{\rho^m}{m!} \int_{-\infty}^\infty x^{k-1}P_m(\sigma;x)d\psi^{[\cc]}(\sigma;x) 
\end{eqnarray*}
and then
$$
   \E\left( z^{\kappa^{[\cc]}(\sigma)}  \right) = \frac{\sigma  }{\sigma+\rho}\sum_{m=0}^\cc \frac{\rho^m}{m!} \int_{-\infty}^\infty \frac{z P_m(\sigma;x)}{1-z x}d\psi^{[\cc]}(\sigma;x) .
$$
By using Equation~\eqref{eqtechcc1}, we have
\begin{align*}
 &   \int_{-\infty}^\infty \frac{z P_m(\sigma;x)}{1-z x}d\psi^{[\cc]}(\sigma;x) =   P_m\left(\sigma;\frac{1}{z}\right)  \frac{P^*_{\cc+1}(\sigma;\frac{1}{z})-z P^*_{\cc}(\sigma;\frac{1}{z})}{P_{\cc+1}(\sigma;\frac{1}{z})-z P_{\cc}(\sigma;\frac{1}{z})} - P^*_m\left(\sigma;\frac{1}{z}\right) \\
 & =(\sigma+\rho) z^{m+1}  \left( - \frac{C^*_{\cc+1}(-Z;\rho z^2) -  C^*_{\cc}(-Z;\rho z^2)    }{C_{\cc+1}(-Z ;\rho z^2) - C_{\cc}(-Z;\rho z^2) } C_m(-Z;\rho z^2) + C^*_m(-Z ;\rho z^2) \right)
\end{align*}
where $Z=\sigma+\rho - \rho z^2$ and we have used the relation between the polynomials $P_m(\sigma,z)$ and the Charlier polynomials. Now, Equation~\eqref{eqtechcc2} yields
$$
 \int_{-\infty}^\infty \frac{z P_m(\sigma;x)}{1-z x}d\psi^{[\cc]}(\sigma;x) = (\sigma+\rho)z^{m+1}  \int_{0}^\infty \frac{C_m(x;\rho z^2)}{\sigma+\rho-\rho z^2 x}d\phi^{[\cc]}(\rho z^2;x) 
$$
and Equation~\eqref{genkappacc} follows.
\end{proof}

By using Equation~\eqref{genkappacc}, the generating function of the random variable $K^{[\cc]}(t)$ counting the number of customers entering  or leaving the queue is given by
$$
     \E\left( z^{K^{[\cc]}(t)}  \right) = z \int_{0}^\infty {\mathcal{C}^{[\cc]} (x;\rho z^2,\rho z)}e^{-\rho t(1-z^2 x)}d\phi^{[\cc]}(\rho z^2;x)
$$

Finally, let $\nu^{[\cc]}(\sigma)$ be the number  of the departures from the queue before the observer leaves the system. The generating function of this random variable is given by the following result.

\begin{proposition}
    The generating function of the random variable $\delta^{[\cc]}(\sigma)$ is given by
    \begin{equation}
        \label{gendeltacc}
        \E\left( z^{\delta^{[\cc]}(\sigma)}  \right) = \frac{\sigma}{\sigma+\rho}  \int_{-\infty}^\infty \frac{\mathcal{C}^{[\cc]} \left(\frac{\rho-(\sigma+\rho)x^2}{x^2}; \frac{\rho}{x^2};\rho\right)}{1 -z x^2}d\psi^{[\cc]}(\sigma;x) ,
    \end{equation}
 where $ \mathcal{C}^{[\cc]} (x;a;z)$ is defined by Equation~\eqref{defCcc}
\end{proposition}

\begin{proof}
 From Section~\ref{transient}, we have
\begin{eqnarray*}
 \P(\mathfrak{d}^{[\cc]}(\sigma)=k) &=&  \sum_{m=0}^\cc \frac{\sigma}{m+\sigma+\rho} {}^te_0 \left(A^{[\cc]}(\sigma)\right)^{2k+m} e_m\\
&=&\frac{\sigma}{\sigma+\rho}\sum_{m=0}^\cc \frac{\rho^m}{m!} \int_{-\infty}^\infty x^{2k+m}P_m(\sigma;x)d\psi^{[\cc]}(\sigma;x).
\end{eqnarray*}   
By using the relationship between polynomials $P_n(\sigma;x)$ and Charlier polynomials, Equation~\eqref{gendeltacc} follows.
\end{proof}

\section{Conclusion}
\label{conclusion}

We have analyzed in this paper some transient characteristics of  an initially empty  $M/M/\infty$ system over a  finite time interval $(0,t)$ via spectral theory, notably the number of departures from the queue. The probability mass function of these random variables can be obtained by using probabilistic arguments. Nevertheless, the use of spectral theory is more systematic in the sense that the same framework can be applied to other models, which may be not amenable via probabilistic analysis. We have illustrated this point by considering the finite capacity $M/M/\cc/\cc$ system. Other models such as the $M/M\cc/\infty$ analyzed in \cite{KMcG} can be analyzed via spectral theory.



}

\bibliographystyle{plain}
\bibliography{adhoc_revised}

\begin{thebibliography}{10}

\bibitem{Abramowitz}
M.~Abramowitz and I.~Stegun.
\newblock {\em Handbook of Mathematical Functions}.
\newblock Dover Publications, 1965.

\bibitem{ismail}
R.~Askey and M.~Ismail.
\newblock {\em Recurrence relations, continued fractions and orthogonal
  polynomials}, volume 49 No 300.
\newblock Memoirs of the American Mathematical Society, 1984.

\bibitem{erdelyi}
Harry Bateman and Arthur Erdélyi.
\newblock {\em {Higher transcendental functions}}.
\newblock California Institute of technology. Bateman Manuscript project.
  McGraw-Hill, New York, NY, 1955.

\bibitem{polya}
C.E. Billigheimer, G.~Polya, and G.~Szeg{\"o}.
\newblock {\em Problems and Theorems in Analysis II: Theory of Functions.
  Zeros. Polynomials. Determinants. Number Theory. Geometry}.
\newblock Classics in Mathematics. Springer Berlin Heidelberg, 1997.

\bibitem{chihara}
T.S. Chihara.
\newblock {\em An Introduction to Orthogonal Polynomials}.
\newblock Dover Books on Mathematics. Dover Publications, 2011.

\bibitem{flajoletlp}
P.~Flajolet.
\newblock Combinatorial aspects of continued fractions.
\newblock {\em Discrete Mathematics}, 32(2):125--161, 1980.

\bibitem{flajoletbdp}
Philippe Flajolet and Fabrice Guillemin.
\newblock The formal theory of birth-and-death processes, lattice path
  combinatorics and continued fractions.
\newblock {\em Advances in Applied Probability}, 32(3):750–778, 2000.

\bibitem{grossharris}
Donald Gross and Carl~M. Harris.
\newblock {\em Fundamentals of Queueing Theory (2nd Ed.).}
\newblock John Wiley \& Sons, Inc., New York, NY, USA, 1985.

\bibitem{FGBJ}
F.~Guillemin and J.~Boyer.
\newblock Analysis of the {M/M/1} queue with processor sharing via spectral
  analysis.
\newblock {\em Queueing Systems}, 39:377 -- 397, 2001.

\bibitem{guillemin_pinchon_1998}
Fabrice Guillemin and Didier Pinchon.
\newblock {Continued Fraction Analysis of the Duration of an Excursion in an $
  M/M/\infty$ System}.
\newblock {\em Journal of Applied Probability}, 35(1):165–183, 1998.

\bibitem{guillemin_pinchon_1999}
Fabrice Guillemin and Didier Pinchon.
\newblock {On a random variable associated with excursions in an $M/M/\infty$
  System}.
\newblock {\em Queueing Systems}, 32(4), 1999.

\bibitem{guillemin_quintuna20}
Fabrice Guillemin and Veronica~Quintuna Rodriguez.
\newblock On the sojourn of an arbitrary customer in an {$M/M/1$} processor
  sharing queue.
\newblock {\em Stochastic Models}, 36(3):378--400, 2020.

\bibitem{FGAS}
Fabrice Guillemin and Alain Simonian.
\newblock {Transient characteristics of an $M/M/\infty$ system}.
\newblock {\em Advances in Applied Probability}, 27(3):862–888, 1995.

\bibitem{Henrici}
P.~Henrici.
\newblock {\em Applied and computational complex analysis}, volume~II.
\newblock John Wiley and Sons, 1991.

\bibitem{KMcGBDP}
S.~Karlin and J.~Mc Gregor.
\newblock {The differential equation of birth and death processes, and the
  Stieltjes moment problem}.
\newblock {\em Advances in Applied Probability}, 85:489–546, 1957.

\bibitem{KMcG}
Samuel Karlin and James McGregor.
\newblock Many server queueing processes with poisson input and exponential
  service times.
\newblock {\em Pacific J. Math.}, 8(1):87--118, 1958.

\bibitem{Klein0}
L.~Kleinrock.
\newblock {\em Queueing Systems}, volume~1.
\newblock Wiley, New York, 1976.

\bibitem{MorrisonMMinf}
John~A. Morrison, Larry~A. Shepp, and Christopher~J. Van~Wyk.
\newblock A queueing analysis of hashing with lazy deletion.
\newblock {\em SIAM Journal on Computing}, 16(6):1155--1164, 1987.

\bibitem{reed}
M.~Reed and B.~Simon.
\newblock {\em Methods of modern physics: Fourier analysis, selfadjointness},
  volume~II.
\newblock Elsevier (Singapore), 2003.

\bibitem{Ross97}
Sheldon~M. Ross.
\newblock {\em Introduction to Probability Models}.
\newblock Academic Press, San Diego, CA, USA, sixth edition, 1997.

\end{thebibliography}

\end{document}